\documentclass[12pt,a4paper]{amsart}
\usepackage{hyperref,amsmath,amsfonts,mathrsfs}
\usepackage{graphics}
\newtheorem{theorem}{Theorem}[section]

\newtheorem{corollary}[theorem]{Corollary}
\theoremstyle{theorem}

\newtheorem{lemma}[theorem]{Lemma}
\theoremstyle{definition}

\numberwithin{equation}{section} \makeatletter
\@namedef{subjclassname@2010}{ 2010 Mathematics Subject
Classification} \makeatother
\begin{document}
\title[Energy, Laplacian energy of double graphs and new families of equienergetic graphs]{Energy, Laplacian energy of double graphs and new families of equienergetic graphs}
\author{S. Pirzada and Hilal A Ganie}

\address{Department of Mathematics \\
    University of Kashmir \\
   Srinagar, Hazratbal 190006\
   \\ India}
 \email{pirzadasd@kashmiruniversity.ac.in; sdpirzada@yahoo.co.in}
 \email{hilahmad1119kt@gmail.com}
\date{}
%\thanks{(P.Gochhayat:Corresponding author) The present investigation of the second author is supported
%under the Fast Track Research Project for Young Scientist,
%Department of Science and Technology, New Delhi, Government of
%India. Sanction Letter No. 100/IFD/12100/2010-11.}\biguplus

\begin{abstract}
For a graph $G$ with vertex set $V(G)=\{v_1, v_2, \cdots, v_n\}$, the extended double cover $G^*$ is a bipartite graph with bipartition (X, Y), $X=\{x_1, x_2, \cdots, x_n\}$ and $Y=\{y_1, y_2, \cdots, y_n\}$, where two vertices $x_i$ and $y_j$ are adjacent if and only if $i=j$ or $v_i$ adjacent to $v_j$ in $G$. The double graph $D[G]$ of $G$ is a graph obtained by taking two copies of $G$ and joining each vertex in one copy with the neighbours of corresponding vertex in another copy. In this paper we study energy and Laplacian energy of the graphs $G^*$ and $D[G]$, $L$-spectra of $G^{k*}$ the $k$-th iterated extended double cover of $G$. We obtain a formula for the number of spanning trees of $G^*$. We also obtain some new families of equienergetic and $L$-equienergetic graphs.
\end{abstract}
\maketitle{}
\section{introduction}\label{sec1}
Let $G$ be finite, undirected, simple graph with $n$ vertices and $m$ edges having vertex set $V(G)=\{v_1, v_2, \cdots, v_n\}$. Throughout this paper we denote such a graph by $G(n,m)$. The adjacency matrix $A=(a_{ij})$ of $G$ is a $(0, 1)$-square matrix of order $n$ whose $(i,j)$-entry is equal to one if $v_i$ is adjacent to $v_j$ and equal to zero, otherwise. The spectrum of the adjacency matrix is called the $A$-spectrum of $G$. If $\lambda_1, \lambda_2, \cdots, \lambda_n$ is the adjacency spectrum of $G$, the energy of $G$ is defined as $E(G)=\sum_{i=1}^{n}|\lambda_i|$.  This quantity introduced by I. Gutman has noteworthy chemical applications (see \cite{gp}).\\
\indent Let $D(G)={diag}(d_1, d_2, \cdots, d_n)$ be the diagonal matrix associated to $G$, where $d_i$ is the degree of vertex $v_i$. The matrices $L(G)$=$D(G)$-$A(G)$ and $L^+(G)$=$D(G)$+$A(G)$ are called Laplacian and signless Laplacian matrices and their spectras are respectively called Laplacian spectrum ($L$-spectrum) and signless Laplacian spectrum ($Q$-spectrum) of $G$. Being real symmetric, positive semi-definite matrices, let $0=\mu_n\leq\mu_{n-1}\leq\cdots\leq\mu_1$ and $0\leq\mu^+_n\leq\mu^+_{n-1}\leq\cdots\leq\mu^+_1$ be respectively the $L$-spectrum and $Q$-spectrum of $G$. It is well known that $\mu_n$=0 with multiplicity equal to the number of connected components of $G$ (see \cite{f}). In (see\cite{f}) Fiedler showed that a graph $G$ is connected if and only if its second smallest Laplacian eigenvalue is positive and called it as the algebraic connectivity of the graph $G$. Also it is well known that for a bipartite graph the $L$-spectra and $Q$-spectra are same (see \cite{cs}). The Laplacian energy of a graph $G$ as put forward by Gutman and Zhou (see \cite{gz}) is defined as $LE(G)=\sum\limits_{i=1}^{n}|\mu_i-\frac{2m}{n}|$. This quantity, which is an extension of graph-energy concept has found remarkable chemical applications beyond the molecular orbital theory of conjugated molecules (see \cite{rg}). Both energy and Laplacian energy have been extensively studied in the literature (see \cite{a,lsg,wh,z,zg} and the references therein). It is easy to see that
\indent $tr(L(G))=\sum_{i=1}^{n}\mu_i=\sum_{i=1}^{n-1}\mu_i=2m$ and $tr(LE^+(G))=\sum_{i=1}^{n}\mu^+_i=2m$.\\
\indent Two graphs $G_1$ and $G_2$ of same order are said to be equienergetic if $E(G_1)$=$E(G_2)$, (see \cite{b,rgra}). In analogy to this two graphs $G_1$ and $G_2$ of same order are said to $L$-equienergetic if $LE(G_1)$=$LE(G_2)$ and $Q$-equienergetic if $LE^+(G_1)$=$LE^+(G_2)$. Since cospectral (Laplacian cospectral) graphs are always equienergetic ($L$-equienergetic) the problem of constructing equienergetic ($L$-equienergetic) graphs is only considered for non-cospecral (non Laplacian cospectral) graphs.\\
\indent The extended double cover \cite{c} of the graph $G(n,m)$ with vertex set $V(G)=\{v_1, v_2, \cdots, v_n\}$ is  a bipartite graph $G^*$ with bipartition (X, Y), $X=\{x_1, x_2, \cdots, x_n\}$ and $Y=\{y_1, y_2, \cdots, y_n\}$, where two vertices $x_i$ and $y_j$ are adjacent if and only if $i=j$ or $v_i$ adjacent $v_j$ in $G$. It is easy to see that $G^*$ is connected if and only if $G$ is connected and a vertex $v_i$ is of degree $d_i$ in $G$ if and only if it is of degree $d_i+1$ in $G^*$. Also the extended double cover $G^*$ of the graph $G$ always contains a perfect matching. The double graph $D[G]$ of $G$ is a graph obtained by taking two copies of $G$ and joining each vertex in one copy with the neighbours of corresponding vertex in another copy. The $k$-fold graph $D^k[G]$ \cite{ms} of the graph $G$ is obtained by taking $k$ copies of the graph $G$ and joining each vertex in one of the copy with the neighbours of the corresponding vertices in the other copies. If $T_n$ is the graph obtained from the complete graph $K_n$ by adding a loop at each of the vertex, it is easy to see that $D^k[G]=G\otimes T_k$. In this paper we study energy, Laplacian energy of the graphs $G^*$ and $D[G]$, the $L$-spectra of $G^{k*}$ the $k$-th iterated extended double cover of $G$ and obtain a formula for the number of spanning trees of $G^*$. We also obtain some new families of the equienergetic and $L$-equienergetic graphs.\\
\indent  We denote the complement of graph $G$ by $\bar{G}$, the complete graph on $n$ vertices by $K_n$, the empty graph by $\bar{K_n}$ and the complete bipartite graph with cardinalities of partite sets $q$ and $r$ by $K_{q,r}$. The rest of the paper is organised as follows. In section 2, energy of the graphs $G^*$ and $D^k[G]$ are obtained and some new families of equienergetic graphs are given, in section 3 $L$-spectra of $G^{k*}$ and a formula for the number of spanning tress of $G^*$ is obtained and in section 4 Laplacian energy of the graphs $G^*$ and $D^k[G]$ and the construction of some new families of $L$-equienergetic graphs by using the graphs $G^{k*}$ and $D^k[G]$ is presented.

\section{energy of double graphs}

In this section we find the energy of the graphs $G^*$ and $D^k[G]$. We also construct some new families of equienergetic graphs based on these graphs.\\
\indent For the graphs $G_1$ and $G_2$ with disjoint vertex sets $V(G_1)$ and $V(G_2)$, the {\it cartesian product} is a graph $G=G_1\times G_2$ with vertex set $V(G_1)\times V(G_2)$ and an edge $((u_1, v_1), (u_2, v_2))$ if and only if $u_1=u_2$ and $(v_1, v_2)$ is an edge of $G_2$ or $v_1=v_2$
and $(u_1, u_2)$ is an edge of $G_1$. The following result gives the $A$-spectra ($L$-spectra) of the cartesian product of graphs.
\begin{lemma}\cite{cds}
If $G_1(n_1,m_1)$ and $G_2(n_2,m_2)$ are two graphs having $A$-spectra ($L$-spectra) respectively as, $\mu_1, \mu_2, \cdots, \mu_{n_1}$ and $\sigma_1, \sigma_2, \cdots, \sigma_{n_2}$, then the $A$-spectra  ($L$-spectra) of $G=G_1\times G_2$ is $\mu_i+\sigma_j$ where $i=1,2,\cdots,n_1$ and $j=1,2,\cdots,n_2$.
\end{lemma}
\indent The {\it conjunction (Kronecker product)} of $G_1$ and $G_2$ is a graph $G=G_1\otimes G_2$ with vertex set $V(G_1)\times V(G_2)$ and an edge $((u_1,v_1), (u_2, v_2))$ if and only if $(u_1, u_2)$ and $(v_1, v_2)$ are edges in $G_1$ and $G_2$, respectively. The following result gives the $A$-spectra ($L$-spectra) of the Kronecker product of graphs.
\begin{lemma}\cite{cds}
If $G_1(n_1,m_1)$ and $G_2(n_2,m_2)$ are two graphs having $A$-spectra ($L$-spectra) respectively as, $\mu_1, \mu_2, \cdots, \mu_{n_1}$ and $\sigma_1, \sigma_2, \cdots, \sigma_{n_2}$, then the $A$-spectra  ($L$-spectra) of $G=G_1\otimes G_2$ is $\mu_i\sigma_j$ where $i=1,2,\cdots,n_1$ and $j=1,2,\cdots,n_2$.
\end{lemma}
\indent The {\it join (complete product)} of $G_1$ and $G_2$ is a graph $G=G_1\vee G_2$ with vertex set $V(G_1)\cup V(G_2)$ and an edge set consisting of all the edges of $G_1$ and $G_2$ together with the edges joining each vertex of $G_1$ with every vertex of $G_2$. The $L$-spectra of join of graphs is given by the following result.
\begin{lemma}\cite{cds}
If $G_1(n_1,m_1)$ and $G_2(n_2,m_2)$ are two graphs having $L$-spectra respectively as $\mu_1,\mu_2,\cdots,\mu_{n_1-1},\mu_{n_1}=0$ and $\sigma_1, \sigma_2, \cdots, \sigma_{n_2-1}, \sigma_{n_2}=0$, then the $L$-spectra of $G=G_1\vee G_2$ is $n_1+n_2, n_1+\sigma_1, n_1+\sigma_2, \cdots, n_1+\sigma_{n_2-1}, n_2+\mu_1, n_2+\mu_2, \cdots, n_2+\mu_{n_1-1}, 0$.
\end{lemma}
\indent The following result \cite{c} gives the $A$-spectra of $G^*$, the extended double cover of the graph $G$.
\begin{theorem}
If $\lambda_1, \lambda_2, \cdots, \lambda_n$ is the $A$-spectra of a graph $G$, then the $A$-spectra of the graph $G^*$ is $\pm(\lambda_1+1), \pm(\lambda_2+1), \cdots, \pm(\lambda_n+1)$.
\end{theorem}
\indent If $\lambda_1, \lambda_2, \cdots, \lambda_n$ is the $A$-spectra of the graph $G$, then by Lemma 2.1, the $A$-spectra of the graph $G\times K_2$ is $\lambda_i+1, \lambda_i-1$ for $1\leq i \leq n$. It is clear from Theorem 2.4, that the graphs $G\times K_2$ and $G^*$ are cospectral if and only if $G$ is bipartite \cite{c}. If $D^k[G]$ is the $k$-fold graph of the graph $G$, the $A$-spectra of $D^k[G]$ is given by the following result from \cite{ms}.
\begin{theorem}
If $\lambda_1, \lambda_2, \cdots, \lambda_n$ is the $A$-spectra of a graph $G$, then the $A$-spectra of the graph $D^k[G]$ is $k\lambda_1, k\lambda_2, \cdots, k\lambda_n, 0$ ($(k-1)n$ times).
\end{theorem}
\indent If $\lambda_1, \lambda_2, \cdots, \lambda_n$ is the $A$-spectra of the graph $G$, then by Theorem 2.4, the $A$-spectra of the graph $G^*$ is $\pm(\lambda_1+1), \pm(\lambda_2+1), \cdots, \pm(\lambda_n+1)$ and by Theorem 2.5, the $A$-spectra of $D^k[G]$ is $k\lambda_1, k\lambda_2, \cdots, k\lambda_n, 0$ ($(k-1)n$ times). Therefore,
\begin{align*} E(G^*)=\sum\limits_{i=1}^{n}|\lambda_i+1|+\sum\limits_{i=1}^{n}|-\lambda_i-1|=2\sum\limits_{i=1}^{n}|\lambda_i+1|,
\end{align*}
and
\begin{align*}
E(D^k[G])=\sum\limits_{i=1}^{n}|2\lambda_i|=2\sum\limits_{i=1}^{n}|\lambda_|=kE(G).
\end{align*}
\indent If $\lambda_1, \lambda_2, \cdots, \lambda_n$ is the $A$-spectra of a graph $G$, then the $A$-spectra of the graph $(G\otimes K_2)\times K_2$ is $\lambda_i+1, \lambda_i-1, -\lambda_i+1, -\lambda_i-1$, $1\leq i\leq n$. Therefore,
\begin{align*}
E((G\otimes K_2)\times K_2)=2\sum_{i=1}^{n}|\lambda_i+1|+2\sum_{i=1}^{n}|\lambda_i-1|=2\left(\sum_{i=1}^{n}|\lambda_i+1|+\sum_{i=1}^{n}|\lambda_i-1| \right)\\=2E(G\times K_2)=E(2(G\times K_2))=E((G\times K_2)\cup (G\times K_2)).
\end{align*}
\indent From the above discussion, we observe that the graphs $(G\otimes K_2)\times K_2$ and $(G\times K_2)\cup (G\times K_2)$ are equienergetic. Moreover, if the graph $G$ is a bipartite graph then the graphs $(G\otimes K_2)\times K_2$ and $G^*\cup G^*$ are also equienergetic graphs.\\
\indent As seen above $E(D^k[G])=k\sum_{i=1}^{n}|\lambda_i|=kE(G)=E(kG)$=$E(G\cup G\cup\cdots$ $k$ copies). This shows that the graphs $D^k[G]$ and $(G\cup G\cup\cdots$ $k$ copies) are non-cospectral equienergetic. However, we show for any graph $G$ the graphs $D[G]$ and $G\otimes K_2$ are always equienergetic non-cospectral graphs.
\begin{theorem}
If $D[G]$ is the double graph of the graph $G$, then the graphs $G\otimes K_2$ and $D[G]$ are non-cospectral equienergetic graphs.
\end{theorem}
\begin{proof}
Let $\lambda_1, \lambda_2, \cdots, \lambda_n$ be the eigenvalues of the graph $G$, then by Lemma 2.2, the eigenvalues of the graph $G\otimes K_2$ are $\lambda_i, -\lambda_i$ for $1\leq i\leq n$ and by Theorem 2.5 (for $k=2$), the eigenvalues of the graph $D[G]$ are $2\lambda_i,~~ 0$($n$ times) for $1\leq i\leq n$. Therefore,
\begin{align*}
E(G\otimes K_2)=\sum\limits_{i=1}^{n}|\lambda_i|+\sum\limits_{i=1}^{n}|-\lambda_i|=2\sum\limits_{i=1}^{n}|\lambda_i|.
\end{align*}
Also,
\begin{align*}
E(D[G])=\sum\limits_{i=1}^{n}|2\lambda_i|=2\sum\limits_{i=1}^{n}|\lambda_i|.
\end{align*}
\indent Clearly these graphs are non-cospectral, so the result follows.
\end{proof}
\indent In general, if $D^k[G]$ be the $k$-fold graph of the graph $G$, we have the following observation.
\begin{theorem}
if $D^k[G]$ is the $k$-fold graph of the graph $G$, then the graphs $D^k[G]$ and $G\otimes sK_2$ are non-cospectral equienergetic graphs if and only if $k=2^s$.
\end{theorem}
\begin{proof}
If $\lambda_1, \lambda_2, \cdots, \lambda_n$ are the eigenvalues of the graph, then by Lemma 2.2, the eigenvalues of the graph $G\otimes sK_2$ are $\lambda_i$ ($2^{s-1}$ times), $-\lambda_i $($2^{s-1}$ times) for $1\leq i\leq n$ and by Theorem 2.5, the eigenvalues of the graph $D^k[G]$ are $k\lambda_i,~~ 0$ ($(k-1)n$ times) for $1\leq i\leq n$. Therefore,
\begin{align}
E(G\otimes sK_2)=2^{s-1}\sum\limits_{i=1}^{n}|\lambda_i|+2^{s-1}\sum\limits_{i=1}^{n}|-\lambda_i|=2^s\sum\limits_{i=1}^{n}|\lambda_i|.
\end{align}
Also,
\begin{align}
E(D^k[G])=\sum\limits_{i=1}^{n}|k\lambda_i|=k\sum\limits_{i=1}^{n}|\lambda_i|.
\end{align}
From (2.1) and (2.2) it is clear that $E(G\otimes s K_2)=E(D^k[G])$ if and only if $k=2^s$.
\end{proof}
\indent Let $G^{**}$ be the extended double cover of the graph $G^*$. We have the following result.
\begin{theorem}
If $G$ is an $n$-vertex graph, then $E(G^*\otimes K_2)=E(G^{**})$, if $|\lambda_i|\geq 2$, for all non-zero eigenvalues of $G$. Moreover these graphs are non-cospectral with equal number of vertices.
\end{theorem}
\begin{proof}
Let $\lambda_1, \lambda_2, \cdots, \lambda_n$ be the eigenvalues of the graph $G$. By Theorem 2.4, the eigenvalues of the graph $G^*$ are $\lambda_i+1, -(\lambda_i+1)$ for $1\leq i\leq n$ and so of $G^{**}$ are $\lambda_i+2, \lambda_i, -(\lambda_i+2), -\lambda_i$ for $ 1\leq i\leq n$. Also by Lemma 2.2, the eigenvalues of the graph $G^*\otimes K_2$ are $\lambda_i+1, -(\lambda_i+1), \lambda_i+1, -(\lambda_i+1)$ for $ 1\leq i\leq n$. Assume that $|\lambda_i|\geq 2$. Then\\
$$|\lambda_i+1|=\left \{\begin{array}{lr}|\lambda_i|+1, &\mbox{if $\lambda_i\geq 0$}\\
|\lambda_i|-1, &\mbox{if $\lambda_i< 0$}
\end{array} \right.,~~~ |\lambda_i+2|=\left \{\begin{array}{lr}|\lambda_i|+2, &\mbox{if $\lambda_i\geq 0$}\\
|\lambda_i|-2, &\mbox{if $\lambda_i< 0.$}
\end{array} \right.$$
Therefore,
\begin{align*}
E(G^{**})&=2\sum\limits_{i=1}^{n}|\lambda_i+2|+2\sum\limits_{i=1}^{n}|\lambda_i|
=2\left (\sum\limits_{\lambda_i\geq 0}|\lambda_i+2|+\sum\limits_{\lambda_i< 0}|\lambda_i+2|+\sum\limits_{i=1}^{n}|\lambda_i| \right)\\&=2\left (\sum\limits_{\lambda_i\geq 0}|\lambda_i|+2+\sum\limits_{\lambda_i< 0}|\lambda_i|-2+\sum\limits_{i=1}^{n}|\lambda_i| \right)\\&=2\left(\sum\limits_{i=1}^{n}|\lambda_i|+\sum\limits_{i=1}^{n}|\lambda_i|+2(\sum\limits_{\lambda_i\geq 0}1-\sum\limits_{\lambda_i< 0}1) \right)\\&=4\sum\limits_{i=1}^{n}|\lambda_i|+4\theta,
\end{align*}
where $\theta$ is the difference between the number of non-negative and negative eigenvalues of $G$ and
\begin{align*}
 E(G^*\otimes K_2)&=2\left (\sum\limits_{i=1}^{n}|\lambda_i+1|+\sum\limits_{i=1}^{n}|-(\lambda_i+1)| \right)=4\sum\limits_{i=1}^{n}|\lambda_i+1|\\&=4\left (\sum\limits_{\lambda_i\geq 0}|\lambda_i+1|+\sum\limits_{\lambda_i< 0}|\lambda_i+1| \right)=4\left (\sum\limits_{\lambda_i\geq 0}|\lambda_i|+1+\sum\limits_{\lambda_i< 0}|\lambda_i|-1| \right)\\&=4\sum\limits_{i=1}^{n}|\lambda_i|+4\left (\sum\limits_{\lambda_i\geq 0}1-\sum\limits_{\lambda_i< 0}1 \right)=4\sum\limits_{i=1}^{n}|\lambda_i|+4\theta.
 \end{align*}
\indent  Clearly these graphs are non-cospectral with same number of vertices.
\end{proof}
\indent Let $G$ be a bipartite graph. It is well known that the spectra of $G$ is symmetric about the origin, so half of the non-zero eigenvalues of $G$ lies to the left and half lies to the right of origin. Therefore if $G$ is a bipartite graph having all its eigenvalues non-zero, the number of positive and negative eigenvalues of $G$ are same. Keeping this in mind we have the following result.
\begin{theorem}
If $G^*$ is the extended double cover of the bipartite graph $G$, then the graphs $G^*$ and $D[G]$ are non-cospectral equienergetic if and only if $|\lambda_i|\geq 1$ for all $ 1\leq i\leq n$.
\end{theorem}
\begin{proof}
Let $\lambda_1, \lambda_2, \cdots, \lambda_n$ be the eigenvalues of the graph $G$. By Theorem 2.4, the eigenvalues of the graph $G^*$ are $\lambda_i+1, -\lambda_i-1$ for $ 1\leq i\leq n$ and by Theorem 2.5, the eigenvalues of the graph $D[G]$ are $2\lambda_i,~~ 0 $($n$ times) for $1\leq i\leq n$. Suppose that $|\lambda_i|\geq 1$ for $i=1,2,\cdots,n,$ then
$$|\lambda_i+1|=\left \{\begin{array}{lr}|\lambda_i|+1, &\mbox{if $\lambda_i> 0$}\\
|\lambda_i|-1, &\mbox{if $\lambda_i< 0.$}
\end{array} \right.$$
Therefore,
\begin{align*}
E(G^*)&=\sum\limits_{i=1}^{n}|\lambda_i+1|+\sum\limits_{i=1}^{n}|-\lambda_i-1|=2\sum\limits_{i-1}^{n}|\lambda_i+1|\\&
=2\left (\sum\limits_{\lambda_i> 0}|\lambda_i+1|+\sum\limits_{\lambda_i <0}|\lambda_i+1| \right)
=2\left (\sum\limits_{\lambda_i> 0}(|\lambda_i|+1)+\sum\limits_{\lambda_i <0}(|\lambda_i|-1) \right)\\&
=2\left ((\sum\limits_{\lambda_i> 0}|\lambda_i|+\sum\limits_{\lambda_< 0}|\lambda_i|)+(\sum\limits_{\lambda_i> 0}1-\sum\limits_{\lambda_< 0}1) \right)=2\sum\limits_{i=1}^{n}|\lambda_i|=E(D[G]).
\end{align*}
\indent\indent Clearly these graphs are non-cospectral with same number of vertices.\\
Conversely, suppose that the graphs $G^*$ and $D[G]$ are non-cospectral equienergetic. We will show that $|\lambda_i|\geq 1$ for all $1\leq i\leq n$.\\
\indent Assume to the contrary that $|\lambda_i|< 1$ for some $i$. Then for this $i$, $|\lambda_i+1|=\lambda_i+1$. Without loss of generality, suppose that the eigenvalues of $G$ satisfy $|\lambda_i|\geq 1$, for $i=1,2,\cdots,k $ and $|\lambda_i|< 1$, for $i=k+1,k+2,\cdots,n$, since the eigenvalues are real and reordering does not effect the argument. We have the following cases to consider.\\
{\it Case (i)}. If $\lambda_i> 0$ for $ i=1,2,\cdots,k$ and $\lambda_i\geq 0$ for $ i=k+1,k+2,\cdots,n$, then
\begin{align*}
E(G^*)=2\left(\sum\limits_{i=1}^{k}|\lambda_i+1|+\sum\limits_{i=k+1}^{n}|\lambda_i+1| \right)=2\left(\sum\limits_{i=1}^{n}|\lambda_i|+n \right).
\end{align*}
{\it Case (ii)}. If $\lambda_i> 0$ for $ i=1,2,\cdots,k$ and $\lambda_i\leq 0$ for $ i=k+1,k+2,\cdots,n$, then if $\theta_0$ is the number of zero eigenvalues of $G$, we have
\begin{align*}
E(G^*)&=2\left(\sum\limits_{i=1}^{k}|\lambda_i+1|+\sum\limits_{i=k+1}^{n}|\lambda_i+1| \right)=2\left(\sum\limits_{i=1}^{k}(|\lambda_i|+1)+\sum\limits_{i=k+1}^{n}(\lambda_i+1) \right).\\&
> 2\left(\sum\limits_{i=1}^{k}(|\lambda_i|+1)+\sum\limits_{i=k+1}^{n}(|\lambda_i|-1) \right)=2\left(\sum\limits_{i=1}^{n}|\lambda_i|-\theta_0 \right) .
\end{align*}
{\it Case (iii)}. If $\lambda_i< 0$ for $ i=1,2,\cdots,k$ and $\lambda_i\geq 0$ for $ i=k+1,k+2,\cdots,n$, then
\begin{align*}
E(G^*)&=2\left(\sum\limits_{i=1}^{k}|\lambda_i+1|+\sum\limits_{i=k+1}^{n}|\lambda_i+1| \right)=2\left(\sum\limits_{i=1}^{k}(|\lambda_i|-1)+\sum\limits_{i=k+1}^{n}(|\lambda_i|+1) \right).\\&
=2\left(\sum\limits_{i=1}^{n}|\lambda_i|+\theta_0 \right).
\end{align*}
{\it Case (iv)}. If $\lambda_i< 0$ for $ i=1,2,\cdots,k$ and $\lambda_i\leq 0$ for $ i=k+1,k+2,\cdots,n$, then
\begin{align*}
E(G^*)&=2\left(\sum\limits_{i=1}^{k}|\lambda_i+1|+\sum\limits_{i=k+1}^{n}|\lambda_i+1| \right)=2\left(\sum\limits_{i=1}^{k}(|\lambda_i|-1)+\sum\limits_{i=k+1}^{n}(\lambda_i+1) \right).\\&
> 2\left(\sum\limits_{i=1}^{k}(|\lambda_i|-1)+\sum\limits_{i=k+1}^{n}(|\lambda_i|-1) \right)=2\left(\sum\limits_{i=1}^{n}|\lambda_i|-n \right) .
\end{align*}
Clearly in all these cases, we obtain $E(G^*)\neq E(D[G])$, a contradiction. Therefore the result follows.
\end{proof}
\indent We can also prove Theorem 2.9 by using Theorem 2.6 , the fact that the graphs $G^*$ and $G\times K_2$ are cospectral if $G$ is bipartite (Theorem 2 in \cite{c}) and the graphs $G\times K_2$ and $G\otimes K_2$ are equienergetic if an only if $|\lambda_i|\geq 1$ (Theorem 8 in \cite{bva}).

\section{The Laplacian spectra of $G^{k*}$}

Let $G^*$ be the extended double cover of the graph $G$, define $G^{**}=(G^*)^*$, and in general $G^{k*}=(G^{(k-1)*})^*$, $k\geq 1$, called the $k$-th iterated double cover graph of $G$. The $A$-spectra of $G^{k*}$ is given in \cite{c}. Here we obtain the $L$-spectra of the $k$-th iterated extended double cover $G^{k*}$ of the graph $G$. Since the graph $G^{k*}$ is always bipartite for $k\geq 1$, therefore its Laplacian ($L$-spectra) and signless Laplacian ($Q$-spectra) spectra are same.\\
\indent For any complex square matrices $A$ and $B$ of same order, the following observation can be seen in (\cite{fz} p.no.{41}).
\begin{theorem}
If $A$ and $B$ are complex square matrices of same order, then
$$\quad
\begin{vmatrix}
A & B\\
B & A
\end{vmatrix}
=|A+B||A-B|$$
\indent where the symbol $| |$ denotes the determinant of a matrix.
\end{theorem}
\indent We first obtain the $L$-spectra of $G^*$, the extended double cover of the graph $G$, in the following result.
\begin{theorem}
Let $G(n,m)$ be an $n$-vertex graph having Laplacian and signless Laplacian spectra, respectively as $0=\mu_n<\mu_{n-1}\leq\cdots\leq\mu_1$ and $0<\mu^+_n<\mu_{n-1}^+\leq\cdots\leq\mu_1^+$. Then the Laplacian spectra of $G^*$ is $\mu_1, \mu_2, \cdots, \mu_n,~~ \mu_1^{+}+2, \mu_2^{+}+2, \cdots, \mu_n^{+}+2$.
\end{theorem}
\begin{proof}
Let $A(G)$ be the adjacency matrix of the graph $G$. By a suitable relabelling of vertices it can be seen that the adjacency matrix $A(G^*)$ of the graph $G^*$ is
 $$A(G^*)=\quad
\begin{pmatrix}
0 & A(G)+I_n \\
A(G)+I_n & 0
\end{pmatrix}.$$
Let $D(G)$ and $D(G^*)$ be respectively the degree matrices of the graphs $G$ and $G^*$. It is easy to see that
$$D(G^*)=\quad
\begin{pmatrix}
D(G)+I_n & 0 \\
0 & D(G)+I_n
\end{pmatrix}.$$
Therefore, Laplacian matrix $L(G^*)$ of $G^*$ is $$L(G^*)=D(G^*)-A(G^*)=\quad
\begin{pmatrix}
D(G)+I_n & -(A(G)+I_n) \\
-(A(G)+I_n) & D(G)+I_n
\end{pmatrix}.$$
So the Laplacian characteristic polynomial of $G^*$ is $$C_{G^*}(\lambda)= |\lambda I_{2n}-L(G^*)|=
\quad
\begin{vmatrix}
(\lambda-1)I_n-D(G) & A(G)+I_n \\
A(G)+I_n & (\lambda-1)I_n-D(G)
\end{vmatrix}$$
$$=\left |((\lambda-1)I_n-D(G))-(A(G)+I_n)\right |\left |((\lambda-1)I_n-D(G))+(A(G)+I_n) \right|$$
$$=\left |(\lambda-2)I_n-(D(G)+A(G)) \right| \left|\lambda I_n-(D(G)-A(G)) \right|$$
$$= Q_G(\lambda-2) C_G(\lambda).$$
\indent From this the result follows.
\end{proof}

\indent We now obtain the $L$-spectra of $G^{k*}$ as follows.

\begin{theorem}
Let $G(n,m)$ be a graph having $L$-spectra $\mu_i$, and $Q$-spectra $\mu_i^+$, $1 \leq i \leq n$. The $L$-spectra of the graph $G^{k*}$ is  $\mu_i ~~({k \choose 0}$ times), $\mu_i+2 ~~({k-1 \choose 1}$ times), $\mu_i^{+}+2 ~~({k-1 \choose 0}$ times), $\mu_i+4 ~~({k-1 \choose 2}$ times), $\mu^{+}_i+4 ~~({k-1 \choose 1}$ times), $\cdots$, $\mu_i+2(k-2) ~~({k-1 \choose k-2}$ times), $\mu_i^{+}+2(k-2) ~~({k-1 \choose k-3}$ times), $\mu_i+2(k-1)~~({k-1 \choose k-1}$ times), $\mu_i^{+}+2(k-1)~~({k-1 \choose k-2}$ times), $\mu_i^{+}+2k~~({k \choose k}$ times), where $1\leq i \leq n$.
\end{theorem}
\begin{proof}
We prove this result by induction and we use induction on $k$. For $k=1$, the result follows by Theorem 3.2. For $k=2$, we have $G^{2*}=G^{**}$. Let $A(G^*)$ and $A(G^{**})$ be the adjacency matrices respectively of the graphs $G^*$ and $G^{**}$. It is not difficult to see that
$$A(G^{**})=
\quad\begin{pmatrix}
0 & A(G^*)+I_{2n} \\
A(G^*)+I_{2n} & 0
\end{pmatrix}.$$
Let $D(G^*)$ and $D(G^{**})$ be respectively the degree matrices of $G^*$ and $G^{**}$. It can be seen that $$D(G^{**})=
\quad
\begin{pmatrix}
D(G^*)+I_{2n} & 0\\
0 & D(G^*)+I_{2n}
\end{pmatrix}.$$
Therefore the Laplacian matrix of $G^{**}$ is $$L(G^{**})=D(G^{**})-A(G^{**})
=\quad
\begin{pmatrix}
D(G^*)+I_{2n} & -(A(G^*)+I_{2n}) \\
-(A(G^*)+I_{2n}) & D(G^*)+I_{2n}
\end{pmatrix}.$$
So the Laplacian characteristic polynomial of $G^{**}$ is
$$C_{G^{**}}(\lambda)= |\lambda I_{4n}-L(G^{**})|=
\quad
\begin{vmatrix}
(\lambda-1)I_{2n}-D(G^*) & A(G^*)+I_{2n} \\
A(G^*)+I_{2n} & (\lambda-1)I_{2n}-D(G^*)
\end{vmatrix}$$
$$=\left |((\lambda-1)I_{2n}-D(G^*))-(A(G^*)+I_{2n})\right |\left |((\lambda-1)I_{2n}-D(G^*))+(A(G^*)+I_{2n}) \right|$$
$$=\left |(\lambda-2)I_{2n}-(D(G^*)+A(G^*)) \right| \left| \lambda I_{2n}-(D(G^*)-A(G^*)) \right|$$
$$= Q_{G^*}(\lambda-2)C_{G^*}(\lambda).$$
From this it is clear that the $L$-spectra of $G^{**}$ is $\mu_i,~\mu_i+2,~\mu_i^{+}+2,~\mu_i^{+}+4$, for $1 \leq i \leq n$, that is $L$-spectra of $G^{**}$ is $\mu_i$ (${2 \choose 0}$ times), $\mu_i+2$ (${1 \choose 1}$ times), $\mu_i^{+}+2$ (${1 \choose 0}$ times)  and $\mu_i^{+}+4$ (${2 \choose 2}$ times).
Therefore the result is true in this case. Assume that the result is true for $k=s-1$. Then by induction hypothesis the $L$-spectra of $G^{(s-1)*}$ is
 $\mu_i ({s-1 \choose 0}$ times), $\mu_i+2 ({s-2 \choose 1}$ times), $\mu_i^{+}+2 ({s-2 \choose 0}$ times), $\cdots$, $\mu_i+2(s-2) ({s-2 \choose s-2}$ times), $\mu_i^{+}+2(s-2) ({s-2 \choose s-3}$ times), $\mu_i^{+}+2(s-1) ({s-1 \choose s-1}$ times). Now for $k=s$, it can be seen by proceeding as in the case $k=2$ the Laplacian matrix $L(G^{s*})$ of the graph $G^{s*}$ is  $$L(G^{s*})=D(G^{s*})-A(G^{s*})
=\quad
\begin{pmatrix}
D(G^{(s-1)*})+I_{2^{s-1}n} & -(A(G^{(s-1)*})+I_{2^{s-1}n}) \\
-(A(G^{(s-1)*})+I_{2^{s-1}n}) & D(G^{(s-1)*})+I_{2^{s-1}n}
\end{pmatrix}.$$
\indent Therefore, the Laplacian characteristic polynomial of $G^{s*}$ is
$$C_{G^{s*}}(\lambda)=|\lambda I_{2^{s}n}-L(G^{s*})= \quad
\begin{vmatrix}
(\lambda-1)I_{2^{s-1}n}-D(G^{(s-1)*}) & A(G^{(s-1)*})+I_{2^{s-1}n} \\
A(G^{(s-1)*})+I_{2^{s-1}n} & (\lambda-1)I_{2^{s-1}n}-D(G^{(s-1)*})
\end{vmatrix}$$
$$=\left |((\lambda-1)I_{2^{s-1}n}-D(G^{(s-1)*}))-(A(G^{(s-1)*})+I_{2^{s-1}n})\right |\times$$
$$\left |((\lambda-1)I_{2^{s-1}n}-D(G^{(s-1)*}))+(A(G^{(s-1)*})+I_{2^{s-1}n}) \right|$$
$$=\left |(\lambda-2)I_{2^{s-1}n}-(D(G^{(s-1)*})+A(G^{(s-1)*})) \right| \left| \lambda I_{2^{s-1}n}-(D(G^{(s-1)*})-A(G^{(s-1)*})) \right|$$
$$= Q_{G^{(s-1)*}}(\lambda-2)C_{G^{(s-1)*}}(\lambda).$$
Therefore, it follows that the $L$-spectra of the graph $G^{s*}$ is  $\mu_i ({s-1 \choose 0}$ times), $\mu_i+2 ({s-2 \choose 1}$ times), $\mu_i^{+}+2 ({s-2 \choose 0}$ times), $\cdots$, $\mu_i+2(s-2) ({s-2 \choose s-2}$ times), $\mu_i^{+}+2(s-2) ({s-2 \choose s-3}$ times), $\mu_i^{+}+2(s-1) ({s-1 \choose s-1}$ times), $\mu_i+2 ({s-1 \choose 0}$ times), $\mu_i+4 ({s-2 \choose 1}$ times), $\mu_i^{+}+4 ({s-2 \choose 0}$ times), $\cdots$, $\mu_i+2(s-1) ({s-2 \choose s-2}$ times), $\mu_i^{+}+2(s-1) ({s-2 \choose s-3}$ times), $\mu_i^{+}+2s ({s-1 \choose s-1}$ times).\\
\indent Using ${k \choose r}+{k \choose {r-1}}={k+1 \choose r}$, $0\leq r \leq k$ and ${{s-1} \choose 0}={s \choose 0}={{s-1} \choose {s-1}}={{s-2} \choose {s-2}}=1$, we see that the $L$-spectra of $G^{s*}$ is  $\mu_i ({s \choose 0}$ times), $\mu_i+2 ({s-1 \choose 1}$ times), $\mu_i^{+}+2 ({s-1 \choose 0}$ times), $\mu_i+4 ({s-1 \choose 2}$ times), $\mu_i^{+}+4 ({s-1 \choose 1}$ times), $\cdots$, $\mu_i+2(s-2) ({s-1 \choose s-2}$ times), $\mu_i^{+}+2(s-2) ({s-1 \choose s-3}$ times), $\mu_i+2(s-1) ({s-1 \choose s-1}$ times), $\mu_i^{+}+2(s-1) ({s-1 \choose s-2}$ times), $\mu_i^{+}+2s ({s \choose s}$ times). Thus the result is true in this case as well hence by induction the result follows.
\end{proof}

\indent If $G$ is a bipartite graph, it is easy to see that under elementary transformation the Laplacian characteristic polynomial of $G$ coincides with the signless Laplacian characteristic polynomial of $G$. Therefore the Laplacian and signless Laplacian spectra of $G$  are same. We have the following observation.

\begin{corollary}
If $G(n,m)$ is a bipartite graph having $L$-spectra $\mu_i$, $1\leq i\leq n$, then the $L$-spectra of $k$-th iterated double cover $G^{k*}$ of $G$ is  $\mu_i ({k \choose 0}$ times), $\mu_i+2  ({k \choose 1}$ times), $\cdots$, $\mu_i+2(k-2) ({k \choose k-2}$ times), $\mu_i+2(k-1) ({k \choose k-1}$ times), $\mu_i+2k ({k \choose k}$ times), where $1\leq i \leq n$.
\end{corollary}
\begin{proof}
Since for a bipartite graph $G$ the Laplacian and the signless Laplacian spectra are same, we have $\mu_i=\mu_i^{+}$ for all $  1\leq i \leq n$. Using this in Theorem 3.3, we obtain the $L$-spectra of $G^{k*}$ as  $\mu_i ({k \choose 0}$ times), $\mu_i+2 ({k-1 \choose 1}$ times), $\mu_i+2 ({k-1 \choose 0}$ times), $\mu_i+4 ({k-1 \choose 2}$ times), $\mu_i+4 ({k-1 \choose 1}$ times), $\cdots$, $\mu_i+2(k-2) ({k-1 \choose k-2}$ times), $\mu_i+2(k-2) ({k-1 \choose k-3}$ times), $\mu_i+2(k-1) ({k-1 \choose k-1}$ times), $\mu_i+2(k-1) ({k-1 \choose k-2}$ times), $\mu_i+2k ({k \choose k}$ times). Now using the fact ${t \choose r}+{t \choose {r-1}}={{t+1} \choose r}$, $0\leq r\leq t$, the result follows.
\end{proof}
\indent In \cite{c} three formulae are given for the number of spanning trees of $G^*$ in terms of $A$-spectra of the corresponding graph $G$. We now obtain a formula for the number of spanning trees in terms of the $L$ and $Q$-spectra of $G^*$.
\begin{theorem}
The number of spanning trees $\tau(G^*)$ of the graph $G^*$ is
$$\tau(G^*)=\frac{1}{2}\tau(G)\prod\limits_{i=1}^{n}(\mu_i^{+}+2).$$
\end{theorem}
\begin{proof}
Let $0=\mu_n<\mu_{n-1}\leq\cdots\leq\mu_1$ and $0<\mu^+_n<\mu^+_{n-1}\leq\cdots\leq\mu^+_1$  be respectively the $L$-spectra and the $Q$-spectra of the graph $G$. By Theorem 3.2, the $L$-spectra of the graph $G^*$ is $\mu_i, \mu_i^{+}+2$ for $i=1,2,\cdots,n$. By using the fact that the number of spanning trees of a graph of order $n$ is $\frac{1}{n}$ times the product of $(n-1)$ largest Laplacian eigenvalues of the graph, we have
\begin{align*}
\tau(G^*)=\frac{1}{2n}\prod\limits_{i=1}^{n-1}\mu_i\prod\limits_{i=1}^{n}(\mu_i^{+}+2)=\frac{1}{2}\tau(G)\prod\limits_{i=1}^{n}(\mu_i^{+}+2).
\end{align*}
\end{proof}
\indent In case $G$ is bipartite, $\mu_i=\mu_i^+$, so we have
\begin{align*}
\tau(G^*)=\frac{1}{2n}\prod\limits_{i=1}^{n-1}\mu_i\prod\limits_{i=1}^{n}(\mu_i+2)=\tau(G)\prod\limits_{i=1}^{n-1}(\mu_i+2).
\end{align*}
\indent In \cite{c} it is shown that the graphs $G^*$ and $G\times K_2$ are $A$-cospectral if and only if $G=K_1$ or $G$ is bipartite. An analogous result holds for the $L$-spectra and is given below.
\begin{theorem}
The graphs $G^*$ and $G\times K_2$ are $L$-cospectral if and only if $G=K_1$ or $G$ is bipartite.
\end{theorem}
\begin{proof}
If $G=K_1$, the graphs $G^*$ and $G\times K_2$ are both isomorphic to $K_1$, so are $L$-cospectral. Now if $G\neq K_1$, assume that $G$ is bipartite. Then $\mu_i=\mu_i^+$ and so the $L$-spectra of $G^*$ is $\mu_i, \mu_i+2$ for $1\leq i\leq n$ which is same as the $L$-spectra of $G\times K_2$. Conversely, suppose that the graphs $G^*$ and $G\times K_2$ are $L$-cospectral. Then $\mu_i=\mu_i^+$, which is only possible if $G$ is bipartite. Hence the result.
\end{proof}
\indent An {\it integral graph} is a graph all of whose eigenvalues are integers. Following observation is a consequence of Theorem 3.3.
\begin{theorem}
A graph $G$ is {\it Laplacian integral} if and only if the graph $G^{k*}$ is Laplacian integral graph.
\end{theorem}
\indent It is clear from Theorem 3.7, that given a Laplacian integral $G$ it is always possible to construct an infinite sequence of Laplacian integral graphs. Indeed the graph $G^{k*}$ is Laplacian integral for all $k\geq 1$.\\
\indent Two graphs $G_1$ and $G_2$  are said to be co-spectral, if they are non-isomorphic and have the same spectra. We have the following result, which follows by Theorem 3.3.
\begin{theorem}
Two graphs $G_1$ and $G_2$ are Laplacian co-spectral if and only if the graphs $G_1^{k*}$ and $G_2^{k*}$ are Laplacian co-spectral.
\end{theorem}

\indent Thus given two Laplacian co-spectral graphs $G_1$ and $G_2$, it is always possible to construct an infinite sequence of Laplacian co-spectral graphs. Indeed the graphs $G_1^{k*}$ and $G_2^{k*}$ are Laplacian co-spectral for all $k\geq 1$.\\
\indent Since the extended double cover $G^*$ of the graph $G$ is always bipartite, it follows by Theorem 3.6, the graphs $G^{**}$ and $G^*\times K_2$ are $L$-cospectral and in general the graphs $G^{s*}$ and $G^{(s-1)*}\times K_2$ are $L$-cospectral. Also it is easy to see that the graphs $(G\times K_2)^*$ and $G^*\times K_2$ are $L$-cospectral and in general the graphs $(G\times K_2)^{s*}$ and $G^{s*}\times K_2$ are both $L$-cospectral as well as $Q$-cospectral. Moreover, if $G$ is bipartite then as seen in Theorem 3.6, the graphs $G^*$ and $G\times K_2$ are $L$-cospectral. Using the same argument it can be seen that the graphs $G^{**}$ and $G\times K_2\times K_2$ are $L$-cospectral if and only if $G$ is bipartite. A repeated use of the argument as used in Theorem 3.6, gives the graphs $G^{s*}$ and $(G\times K_2\times K_2\times\cdots s-times)=(G\times sK_2)=(G\times Q_s)$ are $L$-cospectral if and only if $G$ is bipartite. From this discussion it follows that the graphs $G^{s*}$, $G^{(s-1)*}\times K_2$, $(G\times K_2)^{(s-1)*}$ and $G\times Q_{s-1}$ are mutually non-isomorphic $L$-cospectral graphs if and only $G$ is bipartite, where $Q_n$ is the hypercube.

\section{Laplacian energy of double graphs}

In this section, we study the Laplacian energy of the graphs $D[G]$, $D^k[G]$ and $G^*$. Using these graphs we obtain some new families of non Laplacian cospectral $L$-equienergetic graphs.
Let $D[G]$ and $G^*$ be respectively the double graph and the extended double cover of the graph $G$. Then the Laplacian spectra of the graph $G^*$ is given by Theorem 3.2, and the Laplacian spectra of $D^k[G]$ is given by the following Theorem \cite{ms}.
\begin{theorem}
Let $G$ be a graph with $n$ vertices having degrees $d_1, d_2, \cdots, d_n$ and let $\mu_1, \mu_2, \cdots, \mu_n$ be its Laplacian spectra. Then the Laplacian spectra of $D^k[G]$ is $k\mu_i, kd_i$ $((k-1)n$ times) for $1\leq i\leq n$.
\end{theorem}
Let $\mu_i$ for $1\leq i \leq n$ be the $L$-spectra of the graph $G$. Then by Theorem 3.2, the $L$-spectra of the extended double cover $G^*$ of the graph $G$ is $\mu_i, \mu_i^{+}+2$ for $1\leq i\leq n$. Also the average vertex degree of $G^*$ is $\frac{2m}{n}+1$. Therefore,
\begin{align*}
LE(G^*)=\sum\limits_{i=1}^{n}|\mu_i-\frac{2m}{n}-1|+\sum\limits_{i=1}^{n}|\mu_i^{+}-\frac{2m}{n}+1|.
\end{align*}
Since average vertex degree of $D^k[G]$ is $k\frac{2m}{n}$, we have
\begin{align*}
LE(D^k[G])&=\sum\limits_{i=1}^{n}|k\mu_i-k\frac{2m}{n}|+(k-1)\sum\limits_{i=1}^{n}|kd_i-k\frac{2m}{n}|\\&
=k\sum\limits_{i=1}^{n}|\mu_i-\frac{2m}{n}|+k(k-1)\sum\limits_{i=1}^{n}|d_i-\frac{2m}{n}|\\&
=kLE(G)+k(k-1)\sum\limits_{i=1}^{n}|d_i-\frac{2m}{n}|.
\end{align*}
\indent From this it is clear that $LE(D^k[G])=k LE(G)$, if $G$ is regular. Also, since the $k$-fold graph of a regular graph is regular, it follows, if $G_1$ and $G_2$ are $r$-regular $L$-equienergetic graphs then their $k $-fold graphs $D^k[G_1]$ and $D^k[G_2]$ are also $L$-equienergetic. Let $\pounds(G)$ be the line graph of the graph $G$. It is shown in \cite{rgra} that if $G_1$ and $G_2$ are $r$-regular graphs then their $k$-th, ($k\geq 2$) iterated line graphs $\pounds^k(G_1)$ and $\pounds^k(G_2)$ are always equienergetic and so $L$-equienergetic. Therefore it follows that given any two $r$-regular graphs, we can always construct an infinite family of $L$-equienergetic graphs.\\
\indent In case the given $r$-regular connected graphs are $L$-equienergetic, the $k$-fold graph forms a larger family of $L$-equienergetic graphs than the $k$-th iterated line graph. As an example, consider the $4$-regular graphs $G_1$ and $G_2$ shown in Figure 1 on $9$-vertices. It can be seen that the $L$-spectra of $G_1$ and $G_2$ are respectively as $0, 3^4, 6^4$ and $0, 2, 3^2, 5^2, 6^3$ (where $a^s$ means $a$ occurs $s$ times in the spectrum). Therefore $LE(G_1)=16=LE(G_2)$. This shows that the graphs  $G_1$ and $G_2$ are regular $L$-equienergetic graphs, so their $k$-fold graphs $D^k[G_1]$ and $D^k[G_2]$ and their $k$-th, ($k \geq 2$) iterated line graphs are also $L$-equienergetic. Infact the $k$-fold graph gives an infinite family of $L$-equienergetic graph pairs of order $n\equiv 0(\mod 9)$, whereas the $k$-th iterated line graph gives an infinite family of $L$-equienergetic graph pairs of orders $n=54, 270, 2430$, and so on, from this the assertion follows.\\

%TeXCAD Picture [fig1.tex]. Options:
%\grade{\on}
%\emlines{\off}
%\epic{\off}
%\beziermacro{\on}
%\reduce{\on}
%\snapping{\off}
%\quality{8.00}
%\graddiff{0.01}
%\snapasp{1}
%\zoom{4.0000}
\unitlength 1mm % = 2.85pt
\linethickness{0.4pt}
\ifx\plotpoint\undefined\newsavebox{\plotpoint}\fi % GNUPLOT compatibility
\begin{picture}(155.5,66.54)(0,0)
%\emline(6.75,50)(27.25,66.25)
\multiput(6.75,50)(.04253112,.033713693){482}{\line(1,0){.04253112}}
%\end
%\emline(7,50)(17.25,30)
\multiput(7,50)(.033717105,-.065789474){304}{\line(0,-1){.065789474}}
%\end
\put(26.75,65.75){\line(1,0){27.25}}
%\emline(54,65.75)(66,49)
\multiput(54,65.75)(.033707865,-.047050562){356}{\line(0,-1){.047050562}}
%\end
\put(17,30){\line(1,0){31.5}}
%\emline(65.75,49.25)(48.25,30)
\multiput(65.75,49.25)(-.03371869,-.037090559){519}{\line(0,-1){.037090559}}
%\end
%\emline(27.25,65.75)(37,59)
\multiput(27.25,65.75)(.04850746,-.03358209){201}{\line(1,0){.04850746}}
%\end
%\emline(37,59)(54.5,65.75)
\multiput(37,59)(.08706468,.03358209){201}{\line(1,0){.08706468}}
%\end
\put(7.25,49.75){\line(3,-2){12.75}}
\put(20,41.25){\line(-1,-4){2.75}}
%\emline(47.25,42.25)(48.5,30.25)
\multiput(47.25,42.25)(.0328947,-.3157895){38}{\line(0,-1){.3157895}}
%\end
%\emline(47.5,42.25)(65.5,49.25)
\multiput(47.5,42.25)(.08653846,.03365385){208}{\line(1,0){.08653846}}
%\end
%\emline(37.25,58.75)(17.5,30.25)
\multiput(37.25,58.75)(-.033703072,-.048634812){586}{\line(0,-1){.048634812}}
%\end
%\emline(37.25,58.75)(48.5,29.75)
\multiput(37.25,58.75)(.033682635,-.086826347){334}{\line(0,-1){.086826347}}
%\end
%\emline(7,50)(47.25,42)
\multiput(7,50)(.16911765,-.03361345){238}{\line(1,0){.16911765}}
%\end
%\emline(26.75,65.5)(47,42.25)
\multiput(26.75,65.5)(.033693844,-.038685524){601}{\line(0,-1){.038685524}}
%\end
%\emline(19.75,41.5)(54.25,65.5)
\multiput(19.75,41.5)(.0484550562,.0337078652){712}{\line(1,0){.0484550562}}
%\end
%\emline(19.5,41.25)(65.25,49.75)
\multiput(19.5,41.25)(.181547619,.033730159){252}{\line(1,0){.181547619}}
%\end
\put(27,65.5){\circle*{1}}
\put(54,65.25){\circle*{1}}
%\emline(112.5,65.25)(106,57.25)
\multiput(112.5,65.25)(-.03367876,-.04145078){193}{\line(0,-1){.04145078}}
%\end
\put(106,57.25){\line(4,-3){8}}
\put(114,51.25){\line(1,1){7.75}}
%\emline(112.75,65.5)(121.75,59)
\multiput(112.75,65.5)(.04663212,-.03367876){193}{\line(1,0){.04663212}}
%\end
%\emline(114,51.25)(105.5,41.75)
\multiput(114,51.25)(-.033730159,-.037698413){252}{\line(0,-1){.037698413}}
%\end
%\emline(105.5,41.75)(115,35)
\multiput(105.5,41.75)(.04726368,-.03358209){201}{\line(1,0){.04726368}}
%\end
%\emline(114,51.25)(124,44.25)
\multiput(114,51.25)(.04807692,-.03365385){208}{\line(1,0){.04807692}}
%\end
\put(124,44.25){\line(-1,-1){9.25}}
\put(142.75,58.5){\line(0,1){.5}}
\put(142.75,59){\line(0,1){0}}
\put(155.5,59){\line(0,1){0}}
%\emline(146.5,55)(112.75,65.75)
\multiput(146.5,55)(-.105799373,.03369906){319}{\line(-1,0){.105799373}}
%\end
%\emline(113,65.75)(77.5,48.75)
\multiput(113,65.75)(-.070436508,-.033730159){504}{\line(-1,0){.070436508}}
%\end
%\emline(77.75,49)(114.5,35.25)
\multiput(77.75,49)(.090073529,-.03370098){408}{\line(1,0){.090073529}}
%\end
%\emline(106,57)(77.75,49)
\multiput(106,57)(-.11869748,-.03361345){238}{\line(-1,0){.11869748}}
%\end
%\emline(77.75,49)(105.5,41.5)
\multiput(77.75,49)(.12443946,-.03363229){223}{\line(1,0){.12443946}}
%\end
%\emline(121.5,59)(146.25,54.75)
\multiput(121.5,59)(.19642857,-.03373016){126}{\line(1,0){.19642857}}
%\end
\put(155.25,58.75){\line(0,1){0}}
%\emline(123.75,44.25)(145.5,54.75)
\multiput(123.75,44.25)(.069711538,.033653846){312}{\line(1,0){.069711538}}
%\end
\put(65.25,49.25){\circle*{1.41}}
\put(48.5,30){\circle*{1}}
\put(17,30.25){\circle*{1}}
\put(7,49.75){\circle*{1.12}}
\put(19.5,41.25){\circle*{.71}}
\put(37.25,59){\circle*{1}}
\put(47,42.25){\circle*{1.12}}
\put(112.75,65.5){\circle*{1}}
\put(123.5,44.5){\circle*{1.58}}
\put(145.5,54.75){\circle*{1.12}}
\put(78.25,48.5){\circle*{.71}}
\put(71,12.25){\makebox(0,0)[cc]{Figure 1}}
%\emline(114.75,35.25)(114.5,35)
\multiput(114.75,35.25)(-.03125,-.03125){8}{\line(0,-1){.03125}}
%\end
%\emline(114.5,35)(145.5,54.5)
\multiput(114.5,35)(.053633218,.033737024){578}{\line(1,0){.053633218}}
%\end
\put(105.75,41.75){\circle*{.71}}
\put(114.5,35.75){\circle*{1.5}}
\put(121.5,58.75){\circle*{.5}}
\put(106.25,57.25){\circle*{1}}
\put(26.75,65.75){\circle*{1.58}}
\put(53.75,65){\circle*{2.12}}
\put(65,49.25){\circle*{2.24}}
\put(46.75,42){\circle*{1.58}}
\put(48.25,30.25){\circle*{2}}
\put(17,30.25){\circle*{1.12}}
\put(17.25,29.75){\circle*{.71}}
\put(17,30.5){\circle*{1}}
\put(7,49.5){\circle*{1}}
\put(19.25,41.25){\circle*{1.41}}
\put(36.75,59.25){\circle*{1.58}}
\put(19.25,41.5){\circle*{1}}
\put(78.25,49){\circle*{1.41}}
\put(112.75,65.75){\circle*{1.58}}
\put(121,59){\circle*{1.12}}
\put(106.25,57){\circle*{1}}
\put(105.5,41.75){\circle*{1}}
\put(105.25,42){\circle*{1.12}}
\put(78.25,49){\circle*{.71}}
\put(145.5,55){\circle*{1.12}}
\put(114.25,35.5){\circle*{1.12}}
\put(7,49.75){\circle*{1.58}}
%\emline(106,57)(121.5,59)
\multiput(106,57)(.2583333,.0333333){60}{\line(1,0){.2583333}}
%\end
%\emline(121.5,59)(121.25,59.25)
\multiput(121.5,59)(-.03125,.03125){8}{\line(0,1){.03125}}
%\end
%\emline(105.5,42)(123.75,44.5)
\multiput(105.5,42)(.2433333,.0333333){75}{\line(1,0){.2433333}}
%\end
\put(123.75,44.5){\line(0,1){0}}
\put(33.25,22){\makebox(0,0)[cc]{$G_1$}}
\put(114,28.25){\makebox(0,0)[cc]{$G_2$}}
\end{picture}

\indent We have seen that the Laplacian energy of the graph $D[G]$ is twice the Laplacian energy of $G$ when $G$ is regular. But this need not be true for the graph $G^*$ as seen from the Laplacian energy of $G^*$ given above. However we have the following observation.

\begin{theorem}
Let $G^*$ be the extended double cover of the bipartite graph $G$. Then $LE(G^*)=2 LE(G)$ if and only if $|\mu_i-\frac{2m}{n}|\geq 1$ for $1\leq i\leq n$.
\end{theorem}
\begin{proof}
 Let $\mu_i$ for $1\leq i\leq n$ be the $L$-spectra of the graph $G$. Then by Corollary 3.4, the $L$-spectra of $G^*$ is $\mu_i,~~ \mu_i+2$ for $1\leq i\leq n$. Assume that $|\mu_i-\frac{2m}{n}|\geq 1$, for all $i=1,2,\cdots,n$. Then since average vertex degree of $G^*$ is $\frac{2m}{n}+1$, we have
 $$|\mu_i-\frac{2m}{n}+1|=\left \{\begin{array}{lr}|\mu_i-\frac{2m}{n}|+1, &\mbox{if $\mu_i \geq \frac{2m}{n}$}\\
 |\mu_i-\frac{2m}{n}|-1, &\mbox{if $\mu_i<\frac{2m}{n},$}
 \end{array} \right.$$
 and
 $$|\mu_i-\frac{2m}{n}-1|=\left \{\begin{array}{lr}|\mu_i-\frac{2m}{n}|-1, &\mbox{if $\mu_i \geq \frac{2m}{n}$}\\
  |\mu_i-\frac{2m}{n}|+1, &\mbox{if $\mu_i<\frac{2m}{n}.$}
  \end{array} \right.$$
  Therefore,
 \begin{align*}
&LE(G^*)\\&=\sum\limits_{i=1}^{n}|\mu_i-\frac{2m}{n}-1|+\sum\limits_{i=1}^{n}|\mu_i-\frac{2m}{n}+1|
=\sum\limits_{i=1}^{n} \left (|\mu_i-\frac{2m}{n}-1|+|\mu_i-\frac{2m}{n}+1| \right)\\&=\sum_{\mu_i \geq \frac{2m}{n}} \left (|\mu_i-\frac{2m}{n}-1|+|\mu_i-\frac{2m}{n}+1| \right)+\sum_{\mu_i<\frac{2m}{n}} \left (|\mu_i-\frac{2m}{n}-1|+|\mu_i-\frac{2m}{n}+1| \right)\\&
=\sum_{\mu_i \geq \frac{2m}{n}} \left (|\mu_i-\frac{2m}{n}|-1+|\mu_i-\frac{2m}{n}|+1 \right)+\sum_{\mu_i<\frac{2m}{n}} \left (|\mu_i-\frac{2m}{n}|+1+|\mu_i-\frac{2m}{n}|-1 \right)\\&
=2\sum_{\mu_i \geq \frac{2m}{n}} |\mu_i-\frac{2m}{n}| +2\sum_{\mu_i<\frac{2m}{n}} |\mu_i-\frac{2m}{n}|=2 LE(G).
 \end{align*}
 Conversely, suppose that $LE(G^*)=2 LE(G)$. We will show that $|\mu_i-\frac{2m}{n}|\geq 1$ for all $1\leq i\leq n$. We prove this by contradiction. Assume that $|\mu_i-\frac{2m}{n}|< 1$, for some $\lambda_j$. Putting $\beta_i=\mu_i-\frac{2m}{n}$, and using the same argument as used in the converse of Theorem 8 in \cite{bva} we arrive at a contradiction.
 \end{proof}
 \indent If $G$ is a graph satisfying the conditions of Theorem 4.2, then clearly the graphs $G^*$ and $G\cup G$ are $L$-equienergetic. We now obtain some new families of $L$-equienergetic graphs by means of the graphs $G^*$, $G^{k*}$, $D[G]$ and $D^k[G]$.
 \begin{theorem}
 Let $G_1(n,m)$ be a graph having $L$-spectra and $Q$-spectra respectively as $\mu_i$ and $\mu_i^+$ and let $G_2(n,m)$ be another graph having  $L$-spectra and $Q$-spectra respectively as $\lambda_i$ and $\lambda_i^+$ for $i=1,2,\cdots,n$. Then for $p\geq 2n+k$ and $m\leq \frac{(k-1)n}{2}+\frac{k^2}{4}$, $k\geq 3$, we have $LE(G_1^*\vee \bar{K_p})= LE(G_2^*\vee \bar{K_p})$.
 \end{theorem}
 \begin{proof}
 Let $G_1^*$ be the extended double cover of the graph $G_1$. Then by Theorem 3.2, the $L$-spectra of $G_1^*$ is $\mu_i,~~ \mu_i^{+}+2$ for $1\leq i\leq n$ and so by Lemma 2.3, the $L$-spectra of $G_1^*\vee \bar{K_p}$ is $p+2n,~~ p+\mu_i (1\leq i\leq n-1), p+\mu_i^{+}+2 (1\leq i\leq n), 2n$ ($(p-1)$ times), $0$, with average vertex degree $$\frac{2m^{\prime}}{n^{\prime}}=\frac{4m+4pn+2n}{p+2n}.$$ Therefore,
 \begin{align*}
 LE(G_1^*\vee \bar{K_p})&=|p+2n-\frac{2m^{\prime}}{n^{\prime}}|+\sum\limits_{i=1}^{n-1}|p+\mu_i-\frac{2m^{\prime}}{n^{\prime}}|+|0-\frac{2m^{\prime}}{n^{\prime}}|\\&+\sum\limits_{i=1}^{n}|p+\mu_i^{+}+2-\frac{2m^{\prime}}{n^{\prime}}|+(p-1)|2n-\frac{2m^{\prime}}{n^{\prime}}|.
 \end{align*}
 Now, if $p\geq 2n+k$ and $m\leq \frac{(k-1)n}{2}+\frac{k^2}{4}$, $k\geq 3$, we have for $i=1,2,\cdots,n$,
 $$p+\mu_i-\frac{2m^{\prime}}{n^{\prime}}=p+\mu_i-\frac{4m+4pn+2n}{p+2n}=\frac{p(p-2n)+(2n+p)\mu_i-4m-2n}{p+2n}$$ $$\geq \frac{k(2n+k)-2(k-1)n-k^2-2n}{p+2n}=0,$$
 and $$p+\mu_i^{+}+2-\frac{2m^{\prime}}{n^{\prime}}=p+\mu_i^{+}+2-\frac{4m+4pn+2n}{p+2n}=\frac{p(p-2n)+(2n+p)\mu_i^{+}+2(p+n)-4m}{p+2n}$$ $$\geq \frac{k(2n+k)-2(k-1)n-k^2+2(3n+k)}{p+2n}=\frac{8n+2k}{p+2n}\geq 0.$$
 So we have
 \begin{align*}
 LE(G_1^*\vee \bar{K_p})&=(p+2n-\frac{2m^{\prime}}{n^{\prime}})+(n-1)(p-\frac{2m^{\prime}}{n^{\prime}})+n(p+2-\frac{2m^{\prime}}{n^{\prime}})\\&+(p-1)(\frac{2m^{\prime}}{n^{\prime}}-2n)+\frac{2m^{\prime}}{n^{\prime}}+4m\\&
 =6n+(p-2n)\frac{2m^{\prime}}{n^{\prime}}+4m.
 \end{align*}
\indent From this it is clear that the Laplacian energy of $G_1^*$ depends only on the parameters $p, m$ and $n$. Since these parameters are also same for $G_2^*$, it follows that $LE(G_1^*\vee \bar{K_p})= LE(G_2^*\vee \bar{K_p})$.  In fact all the graphs of the family ($G_i^*\vee \bar{K_p}$), $i=1,2,\cdots$, having the same parameters $n, p$ and $m$ satisfying the conditions in the hypothesis are mutually $L$-equienergetic.
 \end{proof}
 \indent Let $G^{t*}$ be the $t$-th iterated extended double cover of the graph $G$. We have the following generalization of Theorem 4.3.
 \begin{theorem}
 Let $G(n,m)$ be a graph having $L$-spectra and $Q$-spectra respectively as $\mu_i$ and $ \mu_i^+$ for $1\leq i\leq n $. For $p\geq 2^tn+k$ and $m\leq \frac{(k-t)n}{2}+\frac{k^2}{2^{t+1}}$, $k\geq t+2, t\geq 1$, we have $LE(G^{t*}\vee \bar{K_p})=2^tn(t+2)+(p-2^tn)\frac{2m^{\prime}}{n^{\prime}}+2^t(2m).$
 \end{theorem}
 \begin{proof}
 Let $G^{t*}$ be the $t$-th iterated extended double cover of the graph $G$. Then by Theorem 3.3, the $L$-spectra of $G^{t*}$ is $\mu_i~~ ({t \choose 0}$ times), $\mu_i+2~~ ({t-1 \choose 1}$ times), $\mu_i^{+}+2~~ ({t-1 \choose 0}$ times), $\mu_i+4 ~~({t-1 \choose 2}$ times), $\mu^{+}_i+4~~ ({t-1 \choose 1}$ times), $\cdots$, $\mu_i+2(t-2)~~ ({t-1 \choose t-2}$ times), $\mu_i^{+}+2(t-2)~~ ({t-1 \choose t-3}$ times), $\mu_i+2(t-1)~~ ({t-1 \choose t-1}$ times), $\mu_i^{+}+2(t-1)~~ ({t-1 \choose t-2}$ times), $\mu_i^{+}+2t ~~({t \choose t}$ times), where $1\leq i \leq n$. So by Lemma 2.3, the $L$-spectra of $G^{t*}\vee \bar{K_p}$ is $0, p+2^tn, 2^tn~~ (p-1$ times), $p+\mu_i ~~({t \choose 0}$ times) $(1\leq i\leq n-1)$, $p+\mu_i+2 ~~({t-1 \choose 1}$ times), $p+\mu_i^{+}+2 ~~({t-1 \choose 0}$ times), $p+\mu_i+4 ~~({t-1 \choose 2}$ times), $p+\mu^{+}_i+4 ~~({t-1 \choose 1}$ times), $\cdots$, $p+\mu_i+2(t-2) ~~({t-1 \choose t-2}$ times), $p+\mu_i^{+}+2(t-2) ~~({t-1 \choose t-3}$ times), $p+\mu_i+2(t-1) ~~({t-1 \choose t-1}$ times), $p+\mu_i^{+}+2(t-1) ~~({t-1 \choose t-2}$ times), $p+\mu_i^{+}+2t ({t \choose t}$ times),  $1\leq i \leq n$, with average vertex degree $$\frac{2m^{\prime}}{n^{\prime}}=\frac{2^{t+1}m+2^ttn+2^{t+1}pn}{p+2^tn}.$$
 Therefore,
 \begin{align*}
 LE(G^{t*}\vee \bar{K_p})&=\sum\limits_{i=1}^{n-1}|p+\mu_i-\frac{2m^{\prime}}{n^{\prime}}|+\sum\limits_{r=1}^{t-1}\sum\limits_{i=1}^{n}{t-1 \choose r}|p+\mu_i+2r-\frac{2m^{\prime}}{n^{\prime}}|\\&+\sum\limits_{r=1}^{t-1}\sum\limits_{i=1}^{n}{t-1 \choose r-1}|p+\mu_i^{+}+2r-\frac{2m^{\prime}}{n^{\prime}}|+\sum\limits_{i=1}^{n}|p+\mu_i+2t-\frac{2m^{\prime}}{n^{\prime}}|\\&+|p+2^tn-\frac{2m^{\prime}}{n^{\prime}}|+(p-1)|2^tn-\frac{2m^{\prime}}{n^{\prime}}|+|0-\frac{2m^{\prime}}{n^{\prime}}|.
 \end{align*}
 Now, if $p\geq 2^tn+k$ and $m\leq \frac{(k-t)n}{2}+\frac{k^2}{2^{t+1}}$, $k\geq t+2, t\geq 1$, we have for $i=1,2,\cdots,n$ and $r=0,1,\cdots,t$
 $$p+\mu_i+2r-\frac{2m^{\prime}}{n^{\prime}}=p+\mu_i+2r-\frac{2^{t+1}m+2^ttn+2^{t+1}pn}{p+2^tn}$$
 $$= \frac{p(p-2^tn)+2r(p+2^tn)+(p+2^tn)\mu_i-2^{t+1}m-2^ttn}{p+2^tn}$$ $$\geq \frac{k(2^tn+k)-k(2^tn+k)+2^ttn-2^ttn}{p+2^tn}=0.$$
Similarly, it can be seen that   $p+\mu_i^{+}+2r-\frac{2m^{\prime}}{n^{\prime}}\geq 0.$
So we have
\begin{align*}
&LE(G^{t*}\vee \bar{K_p})\\&=(n-1)(p-\frac{2m^{\prime}}{n^{\prime}})+\sum\limits_{r=1}^{t-1}\left(n(p+2r-\frac{2m^{\prime}}{n^{\prime}} )+2m \right)\left[{t-1 \choose r}+{t-1 \choose r-1} \right]\\&+(p+2^tn-\frac{2m^{\prime}}{n^{\prime}})+(p-1)(\frac{2m^{\prime}}{n^{\prime}}-2^tn)+\left (n(p+2t-\frac{2m^{\prime}}{n^{\prime}})+2m \right)+\frac{2m^{\prime}}{n^{\prime}}+2m\\&
=2^{t+1}n-pn(2^t-1)+(p-n)\frac{2m^{\prime}}{n^{\prime}}+\sum\limits_{r=1}^{t}{t \choose r}\left(n(p+2r-\frac{2m^{\prime}}{n^{\prime}})+2m \right)+2m\\&
=2^{t+1}n-pn(2^t-1)+(p-n)\frac{2m^{\prime}}{n^{\prime}}+n(2^t-1)(p-\frac{2m^{\prime}}{n^{\prime}})+(2^t-1)2m+2^ttn+2m\\&=2^tn(t+2)+(p-2^tn)\frac{2m^{\prime}}{n^{\prime}}+2^t(2m),
\end{align*}
where we have made use of the fact $\left[{t-1 \choose r}+{t-1 \choose r-1} \right]={t \choose r}$ and $\sum\limits_{r=1}^{t}r{t \choose r}=t2^{t-1}.$\\
\indent Clearly the Laplacian energy of the graph $(G^{t*}\vee \bar{K_p})$ depends only on the parameters $p, m, t$ and $n$. Therefore all the graphs of the families $(G_i^{t*}\vee \bar{K_p})$, where $t,i=1,2,\cdots,$ with the same parameters $p, m, t$ and $n$ satisfying the conditions in the hypothesis are mutually $L$-equienergetic.
 \end{proof}
 \indent Theorem 4.4 gives an infinite family of $L$-equienergetic graphs in various ways, firstly fix the value of $t$ and allow $p$ to vary we obtain families of $L$-equienergetic graphs with same $t$, secondly fix the value of $p$ and allow $t$ to vary we obtain families of $L$-equienergetic graphs with same $p$ and so on.
 \begin{corollary}
 Let $G(n,m)$ be a bipartite graph having $L$-spectra $\mu_i$ for $1\leq i\leq n $. For $p\geq 2^tn+k$ and $m\leq \frac{(k-t)n}{2}+\frac{k^2}{2^{t+1}}$, $k\geq t+2, t\geq 1$, we have $LE(G^{t*}\vee \bar{K_p})=2^tn(t+2)+(p-2^tn)\frac{2m^{\prime}}{n^{\prime}}+2^t(2m).$
 \end{corollary}
 \begin{proof}
 The proof follows by using the Corollary 3.4, and the same argument as in Theorem 4.4.
 \end{proof}
 \indent From Theorem 4.4, it is clear if $G_1$ and $G_2$ are any two graphs with the same parameters, then we can always find tripartite graphs $(G_1^*\vee \bar{K_p})$ and $(G_2^*\vee \bar{K_p})$ having the same Laplacian energy.
 Next we show the construction of $L$-equienergetic graphs by means of graphs $D[G]$ and $D^k[G]$.
 \begin{theorem}
Let $D[G]$ be the double graph of the graph $G$. Then, for $p\geq 2n+k$ and $m\leq \frac{k(2n+k)}{8}$, $k\geq 4$, we have $LE(D[G]\vee \bar{K_p})=4n+(p-2n)\frac{2m^{\prime}}{n^{\prime}}+8m.$
 \end{theorem}
 \begin{proof}
 Let $\mu_i$ and $d_i$, for $i=1,2,\cdots,n$ be the $L$-spectra and the degree sequence of the graph $G$, then by Theorem 4.1, (for $k=2$) the $L$-spectra of the graph $D[G]$ is $2\mu_i$ and $ 2d_i$ for $i=1,2,\cdots,n$ and hence by Lemma 2.3, the $L$-spectra of the graph $D[G]\vee \bar{K_p}$ is $p+2n,~~ p+2\mu_i (1\leq i\leq n-1),~~ p+2d_i (1\leq i\leq n)$, $2n$ ($p-1$ times), $ 0,$ with average vertex degree $$\frac{2m^{\prime}}{n^{\prime}}=\frac{8m+4pn}{p+2n}.$$
 So, if $p\geq 2n+k$ and $m\leq \frac{k(2n+k)}{8}$, $k\geq 4$, we have for $i=1,2,\cdots,n$
 $$p+2\mu_i-\frac{2m^{\prime}}{n^{\prime}}=p+2\mu_i-\frac{8m+4pn}{p+2n}$$ $$=\frac{p(p-2n)+2(p+2n)\mu_i-8m}{p+2n}\geq\frac{k(2n+k)-k(2n+k)}{p+2n}=0.$$
 Similarly, we see that $$p+2d_i-\frac{2m^{\prime}}{n^{\prime}}\geq 0.$$
 Therefore,
 \begin{align*}
 LE(D[G]\vee \bar{K_p})&=|p+2\mu_i-\frac{2m^{\prime}}{n^{\prime}}|+\sum\limits_{i=1}^{n-1}|p+2\mu_i-\frac{2m^{\prime}}{n^{\prime}}|+\sum\limits_{i=1}^{n}|p+2d_i-\frac{2m^{\prime}}{n^{\prime}}|\\&+(p-1)|2n-\frac{2m^{\prime}}{n^{\prime}}|+|0-\frac{2m^{\prime}}{n^{\prime}}|\\&
 =4n+(p-2n)\frac{2m^{\prime}}{n^{\prime}}+8m.
 \end{align*}
 \indent Clearly the Laplacian energy of the graph $D[G]\vee \bar{K_p}$ depends only on the parameters $p, m$ and $n$. Therefore all the graphs of the family $(D[G_i]\vee \bar{K_p})$, $i=1,2,\cdots$ with the same parameters $p, m$ and $n$ satisfying the conditions of the Theorem, are mutually $L$-equienergetic.
 \end{proof}
 \indent If $D^k[G]$ is the $k$-fold graph of the graph $G$, we have the following generalization of Theorem 4.6.
 \begin{theorem}
 Let $D^k[G]$ be the $k$-fold graph of the graph $G$. Then for $p\geq kn+t$ and $m\leq \frac{t(kn+t)}{2k^2}$, $t\geq 2k, k\geq 2$, we have $LE(D^k[G]\vee \bar{K_p})=2kn+(p-nk)\frac{2m^{\prime}}{n^{\prime}}+2mk^2.$
 \end{theorem}
 \begin{proof}
 Let $\mu_i$ and $d_i$ for $i=1,2,\cdots,n$ be respectively the $L$-spectra and the degree sequence of the graph $G$. Then by Theorem 4.1, the $L$-spectra of the graph $D^k[G]$ is $k\mu_i,~~ kd_i$ ($(k-1)n$ times) and so by Lemma 2.3, the $L$-spectra of the graph $D^k[G]\vee \bar{K_p}$ is $ p+kn,~~ p+k\mu_i ~~(1\leq i\leq n-1)$, $p+kd_i$ ($(k-1)n$ times) $(1\leq i\leq n)$, $kn$ ($(p-1)$ times), $0,$ with average vertex degree  $$\frac{2m^{\prime}}{n^{\prime}}=\frac{2k^2m+2pkn}{p+kn}.$$
 So, if $p\geq kn+t$ and $m\leq \frac{t(kn+t)}{2k^2}$, $t\geq 2k, k\geq 2$, we have for $i=1,2,\cdots,n$
 $$p+k\mu_i-\frac{2m^{\prime}}{n^{\prime}}=p+k\mu_i-\frac{2k^2m+2pkn}{p+kn}$$
 $$=\frac{p(p-kn)-2k^2m+k(p+kn)\mu_i}{p+kn}\geq \frac{t(kn+t)-t(kn+t)}{p+kn}=0.$$
 Similarly, we see that $$p+kd_i-\frac{2k^2m+2pkn}{p+kn}\geq 0.$$
 Therefore,
 \begin{align*}
 &LE(D^k[G]\vee \bar{K_p})=|p+kn-\frac{2m^{\prime}}{n^{\prime}}|+\sum\limits_{i=1}^{n-1}|p+k\mu_i-\frac{2m^{\prime}}{n^{\prime}}|+
 (k-1)\sum\limits_{i=1}^{n}|p+2d_i-\frac{2m^{\prime}}{n^{\prime}}|\\&+(p-1)|kn-\frac{2m^{\prime}}{n^{\prime}}|+|0-\frac{2m^{\prime}}{n^{\prime}}|
 =2kn+2mk^2+(p-nk)\frac{2m^{\prime}}{n^{\prime}}.
 \end{align*}
 From this it is clear the Laplacian energy of the graph $(D^k[G]\vee \bar{K_p})$ depends on the parameters $p, k, m$ and $n$. Therefore all the graphs of the families $(D^k[G_i]\vee \bar{K_p})$ where $i=1,2,\cdots$, and $k=2,3,\cdots$ having the same parameters $p, m, k$ and $n$ satisfying the conditions of the Theorem, are mutually $L$-equienergetic.
 \end{proof}
 \indent Theorem 4.7, generates families of $L$-equienergetic graphs in various ways. If we allow $p$ to vary and keep $k$ fixed, we obtain an infinite family of $L$-equienergetic graphs with same $k$ and if we allow $k$ to vary and keep $p$ fixed, we obtain an infinite family of $L$-equienergetic graphs with same $p$ and so on.\\
 \indent If $D[G]$ and $G^*$ are respectively the double graph and the extended double cover of the graph $G$, then the following result gives the construction of $L$-equienergetic graphs with different number of edges.
 \begin{theorem}
 Let $G_1(n,m_1)$ and $G_2(n,m_2)$ be two graphs of order $n\equiv 0(\mod 4)$ with $m_2=m_1+\frac{n}{4}$. Then for $p\geq 4n+k$ and $m_2\leq \frac{n(k-2)}{4}+\frac{k^2}{16}$, $k\geq 4$, we have $$LE(D(G_1^*)\vee \bar{K_p})=LE(D(G_2)^*\vee \bar{K_p})$$.
 \end{theorem}
 \begin{proof}
Let $\mu_i$, $d_i$ and $\mu_i^+$ for $i=1,2,\cdots,n$ be respectively the $L$-spectra, degree sequence and $Q$-spectra of $G_1$ and let $\lambda_i$, $d_i^{\prime}$ and $\lambda_i^+$ be the $L$-spectra, degree sequence and $Q$-spectra of the graph $G_2$. Then by Theorems 3.2 and 4.1 and Lemma 2.3, the $L$-spectra of the graphs $D(G_1^*)\vee \bar{K_p}$ and $D(G_2)^*\vee \bar{K_p}$ are respectively as $p+4n,~ p+2\mu_i ~~(1\leq i\leq n-1),~ p+2\mu_i^{+}+4, p+2d_i+2$ ($2$  times) $(1\leq i\leq n), ~4n$ ($(p-1)$ times), $0$ and $p+4n,~ p+2\lambda_i ~~(1\leq i\leq n-1), ~p+2\lambda_i^{+}+4,~ p+2d_i^{\prime}+2$ ($2$ times) $(1\leq i\leq n), ~~4n $($(p-1)$ times), $0$, with average vertex degrees
$$\frac{2m_1^{\prime}}{n^{\prime}}=\frac{16m_1+8n+8pn}{p+4n} ,  \frac{2m_2^{\prime}}{n^{\prime}}=\frac{16m_2+8n+8pn}{p+4n}.$$
Now, if $p\geq 4n+k$ and $m_2\leq \frac{n(k-2)}{4}+\frac{k^2}{16}$, $k\geq 4$, we have for $i=1,2,\cdots,n$
$$p+2\mu_i-\frac{2m_1^{\prime}}{n^{\prime}}=p+2\mu_i-\frac{16m_1+8n+8pn}{p+4n}$$
$$=\frac{p(p-4n)+2(p+4n)\mu_i-16m_1-8n-8pn}{p+4n}\geq \frac{k(4n+k)-4n(k-2)-k^2-8n}{p+4n}=0.$$
Similarly, we can show that $$p+2\mu_i^{+}+4-\frac{2m_1^{\prime}}{n^{\prime}}\geq 0, p+2d_i+2-\frac{2m_1^{\prime}}{n^{\prime}}\geq 0.$$
Therefore,
\begin{align*}
LE(D(G_1^*)\vee \bar{K_p})&=|p+4n-\frac{2m_1^{\prime}}{n^{\prime}}|+\sum\limits_{i=1}^{n-1}|p+2\mu_i-\frac{2m_1^{\prime}}{n^{\prime}}|+\sum\limits_{i=1}^{n}|p+2\mu_i^{+}+4-\frac{2m_1^{\prime}}{n^{\prime}}|\\&+2\sum\limits_{i=1}^{n}|p+2d_i+2-\frac{2m_1^{\prime}}{n^{\prime}}|+(p-1)|4n-\frac{2m_1^{\prime}}{n^{\prime}}|+|0-\frac{2m_1^{\prime}}{n^{\prime}}|\\&
=16n+16m_1+(p-4n)\frac{2m_1^{\prime}}{n^{\prime}}.
\end{align*}
 Proceeding similarly for the graph $D(G_2)^*\vee \bar{K_p}$ it can be seen that
 \begin{align*}
LE(D(G_2)^*\vee \bar{K_p})=12n+16m_2+(p-4n)\frac{2m_2^{\prime}}{n^{\prime}}.
 \end{align*}
 \indent Using the fact $m_2=m_1+\frac{n}{4}$, the result follows.
\end{proof}
\indent Let $D[G_1]$ be the double graph of the graph $G_1(n,m_1)$ and let $G_2^*$ be the extended double cover of the graph $G_2(n,m_2)$, then for $p\geq 2n+k$ and $m_1\leq \frac{k(2n+k)}{8}$, $k\geq 4$, we have from Theorem 4.6
\begin{align}
LE(D[G_1]\vee \bar{K_p})=4n+8m_1+(p-2n)\frac{2m_1^{\prime}}{n^{\prime}}.
\end{align}
Also, for $p\geq 2n+k$ and $m_2\leq \frac{n(k-1)}{2}+\frac{k^2}{4}$, $k\geq 4$, we have by Theorem 4.3
\begin{align}
LE(G_2^*\vee \bar{K_p})=6n+4m_2+(p-2n)\frac{2m_2^{\prime}}{n^{\prime}}.
\end{align}
\indent If we suppose that $4m_1=2m_2+n$, then it follows from (4.1) and (4.2) that $$LE(D[G_1]\vee \bar{K_p})=LE(G_2^*\vee \bar{K_p}).$$
\indent This gives another construction of families of graphs with same Laplacian energy, same number of vertices but different number of edges. Next we give another way of constructing a family of graphs having same number of vertices, same Laplacian energy but different number of edges.
\begin{theorem}
Let $G_1(n,m_1)$ and $G_2(n,m_2)$ be two graphs with $m_2=2m_1$. Then for $p\geq 4n+k$ and $m_2\leq \frac{k(4n+k)}{8}-n$, $k\geq 4$, we have $LE(D(G_1^*)\vee \bar{K_p})=LE(G_2^{**}\vee \bar{K_p})$.
\end{theorem}
\begin{proof}
Let $\mu_i, \mu_i^+$ and $d_i$ for $i=1,2,\cdots,n$ be respectively the $L$-spectra, $Q$-spectra and the degree sequence of the graph $G_1$ and let $\lambda_i$ and $\lambda_i^+$ be the $L$-spectra and $Q$-spectra of the graph $G_2$. Then by Theorems 3.2 and 4.1 and Lemma 2.3, the $L$-spectra of $D(G_1^*)\vee \bar{K_p}$ is $p+4n,~p+2\mu_i ~~(1\leq i\leq n-1), ~p+2\mu_i^{+}+4, ~p+2d_i+2$ ($2$ times) $(1\leq i\leq n), ~4n $ ($(p-1)$ times), $0$. Also by Theorem 3.3 and Lemma 2.3, the $L$-spectra of the graph $G_2^{**}\vee \bar{K_p}$ is $p+4n, ~p+\lambda_i ~~(1\leq i\leq n-1), ~p+\lambda_i+2, ~p+\lambda_i^{+}+2, ~p+\lambda_i^{+}+4 ~~(1\leq i\leq n), ~4n$($(p-1)$ times), $0$, with average vertex degrees
$$\frac{2m_1^{\prime}}{n^{\prime}}=\frac{16m_1+8n+8pn}{p+4n} , \frac{2m_2^{\prime}}{n^{\prime}}=\frac{8m_2+8n+8pn}{p+4n}.$$
So, if $p\geq 4n+k$ and $m_2\leq \frac{k(4n+k)}{8}-n$, $k\geq 4$, we have for $i=1,2,\cdots,n$
$$p+2\mu_i-\frac{2m_1^{\prime}}{n^{\prime}}=p+2\mu_i-\frac{16m_1+8n+8pn}{p+4n}$$
$$=\frac{p(p-4n)+2(p+4n)\mu_i-16m_1-8n-8pn}{p+4n}=\frac{k(4n+k)-k(4n+k)+8n-8n}{p+4n}=0.$$
Similarly, we can show $$p+2\mu_i^{+}+4-\frac{2m_1^{\prime}}{n^{\prime}}\geq 0, p+2d_i+2-\frac{2m_1^{\prime}}{n^{\prime}}\geq 0.$$
Therefore,
\begin{align*}
LE(D(G_1^*)\vee \bar{K_p})&=|p+4n-\frac{2m_1^{\prime}}{n^{\prime}}|+\sum\limits_{i=1}^{n-1}|p+2\mu_i-\frac{2m_1^{\prime}}{n^{\prime}}|+\sum\limits_{i=1}^{n}|p+2\mu_i^{+}+4-\frac{2m_1^{\prime}}{n^{\prime}}|\\&+2\sum\limits_{i=1}^{n}|p+2d_i+2-\frac{2m_1^{\prime}}{n^{\prime}}|+(p-1)|4n-\frac{2m_1^{\prime}}{n^{\prime}}|+|0-\frac{2m_1^{\prime}}{n^{\prime}}|\\&
=16n+(p-4n)\frac{2m_1^{\prime}}{n^{\prime}}+16m_1.
\end{align*}
\indent Proceeding similarly as above for the graph $G_2^{**}\vee \bar{K_p}$, we can see that
\begin{align*}
LE(G_2^{**}\vee \bar{K_p})=16n+(p-4n)\frac{2m_2^{\prime}}{n^{\prime}}+8m_2.
\end{align*}
\indent Using $m_2=2m_1$, the result follows.
\end{proof}
\indent Theorem 4.9 generates $L$-equienergetic graphs with same number of vertices but different number of edges, infact when one graph contains twice the number of edges as contained in other. Lastly we give the construction of family of graphs with same number of vertices, edges and Laplacian energy by means of cartesian product and extended double cover.
\begin{theorem}
Let $G_1(n,m)$ and $G_2(n,m)$ be two connected non-bipartite graphs. Then for $p\geq n+2$, and $\min(\mu_n^+, \lambda_n^+)\geq \frac{2m}{n}-2$ we have $LE(G_1^*\times K_p)=LE(G_2^*\times K_p) $ if and only if $LE(G_1)=LE(G_2)$.
\end{theorem}
\begin{proof}
Let $0=\mu_n<\mu_{n-1}\leq\cdots\leq \mu_1$ and $0<\mu_n^+<\mu_{n-1}^+\leq\cdots\leq \mu_1^+$ be respectively the $L$-spectra and the $Q$-spectra of the graph $G_1$ and let $0=\lambda_n<\lambda_{n-1}\leq\cdots\leq \lambda_1$ and $0<\lambda_n^+<\lambda_{n-1}^+\leq\cdots\leq \lambda_1^+$ be respectively the $L$-spectra and $Q$-spectra of the graph $G_2$. Then by Theorem 3.2 and Lemma 2.1, the $L$-spectra of the graphs $G_1^*\times K_p$ and $G_1^*\times K_p$ are respectively as $\gamma_i+q_j$ and $\theta_i+q_j$, $i=1,2,\cdots,2n,~~j=1,2,\cdots,n$, where
$$\gamma_i=\left \{\begin{array}{lr}\mu_i, &\mbox{if $i=1,2,\cdots,n$}\\
\mu_i^{+}+2, &\mbox{if $i=n+1,n+2,\cdots,2n,$}
\end{array} \right.$$
$$\theta_i=\left \{\begin{array}{lr}\lambda_i, &\mbox{if $i=1,2,\cdots,n$ }\\
\lambda_i^{+}+2, &\mbox{if $i=n+1,n+2,\cdots,2n$}
\end{array} \right.$$
 and $p=q_1=q_2=\cdots=q_{p-1},~~ q_p=0$
with average vertex degree $$\frac{2m^{\prime}}{n^{\prime}}=\frac{2m}{n}+p.$$
Therefore,
\begin{align*}
LE(G_1^*\times K_p)&=\sum\limits_{i=1}^{2n}\sum\limits_{j=1}^{n}|\gamma_i+q_j-\frac{2m^{\prime}}{n^{\prime}}|\\&=(p-1)\sum\limits_{i=1}^{2n}|p+\gamma_i-\frac{2m^{\prime}}{n^{\prime}}|+\sum\limits_{i=1}^{n}|\gamma_i-\frac{2m_1^{\prime}}{n^{\prime}}|\\&
=(p-1)LE(G_1)+4pn-4n.
\end{align*}
\indent Similarly it can be seen that
\begin{align*}
LE(G_2^*\times K_p)=(p-1)LE(G_2)+4pn-4n.
\end{align*}
 \indent It is now clear that $LE(G_1^*\times K_p)=LE(G_2^*\times K_p) $ if and only if $LE(G_1)=LE(G_2)$, therefore the result follows.
\end{proof}
\indent Since $G^*$ is always bipartite, Theorem 4.10 gives the construction of connected graphs from a given pair of $L$-equienergetic bipartite graphs having same number of vertices, edges and Laplacian energy. Moreover if $t$ is the first value of $p$ satisfying the conditions in Theorem 4.10, then every value greater than $t$ also satisfies this condition, therefore we obtain an infinite family of $L$-equienergetic graph pairs.

\end{document}